\documentclass[reqno,12pt]{amsart}
\usepackage{amscd,amsfonts,mathrsfs,amsthm,amssymb, amsmath,mathtools}
\usepackage{upgreek}
\usepackage{csquotes}
\usepackage{enumerate}
\usepackage{bbm}
\usepackage{multicol}
\usepackage{stmaryrd,color}
\usepackage{array}
\usepackage{ae}
\usepackage{accents}
\usepackage{epsfig}
\usepackage[e]{esvect}
\usepackage[nobysame]{amsrefs}
\usepackage[margin=1in]{geometry}
\usepackage[toc,title,page]{appendix}
\usepackage{hyperref}
\hypersetup{colorlinks=true,linkcolor=blue,filecolor=blue,urlcolor=blue, citecolor=blue}

\let\oldtocsection=\tocsection
\let\oldtocsubsection=\tocsubsection
\let\oldtocsubsubsection=\tocsubsubsection
\renewcommand{\tocsection}[2]{\hspace{0em}\oldtocsection{#1}{#2}}
\renewcommand{\tocsubsection}[2]{\hspace{1em}\oldtocsubsection{#1}{#2}}
\renewcommand{\tocsubsubsection}[2]{\hspace{2em}\oldtocsubsubsection{#1}{#2}}

\def\natural  #1{{\mathbb N^{#1}}}

\def\lap      {\Delta }

\def\equationcolor {\color{black}}
\def\textcolor     {\color{black}}

\def\bcoleq    {\begin{equation}\equationcolor}
\def\ecoleq    {\textcolor\end{equation}}
\def\bcoleqn   {\equationcolor\begin{eqnarray}}
\def\ecoleqn   {\end{eqnarray}\textcolor}

\def\C        {\mathbb{C}}
\def\S        {\mathbb{S}}

\def\gind{\operatorname{g}}

\def\R{\mathbb{R}}

\newcommand{\disp}{\displaystyle}
\newcommand{\eps}{\varepsilon}

\DeclareMathOperator*{\trace}{tr}

\newtheorem{theorem}{Theorem}[section]
\newtheorem{mythm}{Theorem}

\newtheorem{lemma}[theorem]{Lemma}

\newtheorem{proposition}[theorem]{Proposition}

\theoremstyle{definition}
\newtheorem*{assumption*}{$\lambda_{1}$-Condition}

\newtheorem{remark}[theorem]{Remark}

\def\pproof#1{\@ifnextchar[\opargproof
{\opargproof[\it Proof of #1.]}}
\def\opargproof[#1]{\par\noindent {\bf #1 }}

\makeatother

\numberwithin{equation}{section}

\begin{document}

\title[Solitons for the MCF]{On mean curvature flow solitons in the sphere }

\author[M. Magliaro]{\textsc{M. Magliaro}}
\author[L. Mari ]{\textsc{L. Mari}}
\author[F. Roing]{\textsc{F. Roing}}
\author[A. Savas-Halilaj]{\textsc{A. Savas-Halilaj}}

\address{Marco Magliaro \newline
Dipartimento di Scienza e Alta Tecnologia,
Universit\`a degli Studi dell'Insubria,
22100 Como, Italy\newline
{\sl e-mail:} {\bf marco.magliaro@uninsubria.it}
}

\address{Luciano Mari \newline
Dipartimento di Matematica ``Federigo Enriques",
Universit\`a degli Studi di Milano,
20133 Milano, Italy\newline
{\sl e-mail:} {\bf luciano.mari@unimi.it}
}

\address{Fernanda Roing \newline
Dipartimento di Matematica ``Giuseppe Peano",
Universit\`a degli Studi di Torino,
10123 Torino, Italy\newline
{\sl e-mail:} {\bf fernanda.roing@unito.it}
}

\address{Andreas Savas-Halilaj\newline
{Department of Mathematics,
Section of Algebra \!\&\! Geometry, \!\!
University of Ioannina,
45110 Ioannina, Greece} \newline
{\sl E-mail:} {\bf ansavas@uoi.gr}
}

\renewcommand{\subjclassname}{  \textup{2000} Mathematics Subject Classification}
\thanks{A. Savas-Halilaj is supported by (HFRI) Grant No:14758. L. Mari is supported by the PRIN 20225J97H5 ``Differential-geometric aspects of manifolds via Global Analysis"}
\parindent = 0 mm
\hfuzz     = 6 pt
\parskip   = 3 mm
\date{}

\begin{abstract}
In this paper, we consider soliton solutions of the mean curvature flow in the unit sphere $\S^{2n+1}$ moving
along the integral curves of the Hopf unit vector field. While such solitons must necessarily be
minimal if compact, we produce a non-minimal, complete example  with topology
$\S^{2n-1} \times \R$. The example wraps around a Clifford torus $\S^{2n-1} \times \S^1$ along each end, it has reflection and rotational symmetry and its mean curvature
changes sign on each end. Indeed, we prove that a complete 2-dimensional soliton with
non-negative
mean curvature outside a compact set must be a covering of a Clifford torus.
Concluding, we obtain a pinching theorem under suitable conditions on the second
fundamental form.
\end{abstract}

\maketitle
\setcounter{tocdepth}{1}
\section{Introduction}

Let $M$ and $N$ be Riemannian manifolds and let $f_t: M\to N$, $t\in(0,T)$,
be a one parameter family of isometric immersions. Then $\{f_t\}_{t\in(0,T)}$
is called a solution of the {\em mean curvature flow} (MCF) if it satisfies the
evolution equation
\begin{equation}\label{MCF}
\dfrac{\partial f_t}{\partial t}= \textbf{\em H}_{f_t},
\end{equation}
where $\textbf{\em H}_{f_t}$ is the unnormalized mean curvature vector field of the immersion
$f_t$. We often say that the immersed submanifolds $f_t(M)$ move by mean curvature flow.

Interesting solutions to \eqref{MCF} include minimal submanifolds, which are stationary for the flow, and submanifolds moving along the integral curves of Killing vector fields. More precisely, assume that $\xi$ is Killing on $N$ with associated flow of isometries $\varphi:N\times\R\to N$, and let $f:M\to N$ be an isometric immersion satisfying 
\begin{equation}\label{solitonsol}
\textbf{\em H}=\xi^{\perp},
\end{equation}
where $\{\cdot\}^{\perp}$ denotes the orthogonal projection on the normal bundle of $M$. Then, for $t \in \R$ the submanifolds $f_t(M)=\varphi(f(M),t)$ move by mean curvature flow, see for example \cite{hunger1,hunger2}. Solutions
to \eqref{solitonsol} are called solitons for the MCF with respect to $\xi$, or just solitons. In particular, these are {\em eternal solutions} to the MCF.
In the case where $N$ is the Euclidean space and $\xi$ is
a constant vector field, then solutions to \eqref{solitonsol} are called {\em translating
solitons} in direction $\xi$.

We focus on the case where $N$ is the round unit sphere. There are several results concerning the convergence of the MCF for submanifolds
in $\mathbb{S}^{n+1}$, see for example \cite{andrews2,andrews,bryan_ivaki_scheuer,huisken,huisken2,nguyen1,baker,pipoli}.
In recent years there has been increasing interest in \emph{compact} ancient and eternal solutions. In particular, we mention that:
\begin{enumerate}
\item[(a)]
Huisken \& Sinestrari \cite[Theorem 6.1]{huisken_sine} proved under suitable pinching conditions that ancient solutions to the MCF are shrinking spherical caps or
equators.
\item[(b)]
Choi \& Mantoulidis \cite[Theorem 1.1]{choi_mantoulidis} proved a gap phenomenon for
ancient mean curvature flows of submanifolds: if the area of $f_t$ is close 
enough to that of a totally geodesic great sphere as $t \to -\infty$, then the flow must be a steady or shrinking 
sphere. In \cite[Corollary 1.5]{choi_mantoulidis}, for surfaces in
$\mathbb{S}^3$ they obtained a stronger result recovering the case of a steady or shrinking
Clifford torus, provided that the area of $f_t$ does not exceed $2\pi^2+\delta$ with $\delta$ small enough.
\item[(c)]
Chen \& Gaspar \cite{chen_gaspar,chen} constructed an eternal solution to the MCF (in Brakke's sense) in $\mathbb{S}^{n+1}$  connecting an equatorial sphere to a Clifford torus
$\S^1\times\S^{n-1}$.
\item[(d)]
Bryan \& Louie \cite{bryan} proved that the only embedded ancient solutions to the curve
shortening flow on $\S^2$ are equators or shrinking circles, starting at an equator at time
$t=-\infty$ and collapsing to the north pole at time $t=0$.
Bryan, Ivaki \& Scheuer established in \cite{bryan_ivaki_scheuer} analogous results for more general
mean curvature type flows.
\end{enumerate}

In this work we investigate
solutions $f:M^{2n} \to \mathbb{S}^{2n+1}$ to \eqref{solitonsol} in the case where $\xi$ is a unit Killing vector field.
Such fields only exist in odd dimensional spheres and arise from the complex structure of
the corresponding ambient space. More precisely, $\xi$ can be written in the form $J\nu$ where $\nu$ is the unit normal
of $\S^{2n+1}\subset\C^{n+1}$ and $J$ a linear complex structure of $\C^{n+1}$. The
geometry of such vector fields was initiated by Hopf \cite{hopf} and they
are known in the literature as {\em Hopf vector fields}. For this reason we call
{\em Hopf solitons} the solutions of \eqref{solitonsol} when $\xi$ is the Hopf vector field.
It turns out (see Proposition \ref{prop_div_xiT} below) that the scalar mean curvature $H$ of a Hopf soliton satisfies the equation
\[
\operatorname{div}(\xi^{\top})=H^2.
\]
Consequently, by the divergence theorem every compact boundaryless Hopf soliton must be a minimal hypersurface that is tangent to
the Hopf fibration. Among them we underline the minimal Clifford tori
\begin{equation}\label{eq_clifford_gen}
	T_{a,b} = \mathbb{S}^{a} \left( \sqrt{\frac{a}{2n}}\right) \times \mathbb{S}^b \left( \sqrt{\frac{b}{2n}}\right) \subset \R^{a+1} \times \R^{b+1} = \R^{2n+2}, \qquad a+b = 2n,
\end{equation}
appropriately rotated to make $\xi$ tangent to them. In particular, one can prove that the Clifford torus $T_{1,1}$ is  the only complete minimal surface of the $3$-sphere which is tangent to the Hopf vector field; see for instance \cite{pinkall} or the more general Proposition \ref{prop_clifford} below.
However, in higher dimensions, there exist plenty of other examples, which arise by lifting complete minimal
hypersurfaces in the complex projective space via the Hopf fibration. To find non-minimal examples, we shall therefore consider complete, non-compact solitons.
One way to construct examples is by rotating special
curves. As our first main result, in Section \ref{rotationally}, we construct the following example.
\begin{mythm}\label{THMA}
For each $n \ge 1$, there exists a complete, non-minimal Hopf soliton $f : M \to \S^{2n+1}$ which is 
diffeomorphic to $\S^{2n-1} \times \R$ and has the following properties:
\begin{enumerate}
	\item[(a)] $f(M)$ wraps around the same Clifford torus $T_{2n-1,1}$ along each end of $M$, and the mean curvature of $M$ changes sign on each end;
	\item[(b)] there exists an equator $E$ containing a totally geodesic $\mathbb{S}^{2n-1}$ focal to $T_{2n-1,1}$, such that $f(M)$ is symmetric with respect to the reflection in $E$.
\end{enumerate}
\end{mythm}
\begin{remark}\label{rem_wraps}
In the statement of item (a) of Theorem \ref{THMA}, the assertion  ``$f(M)$ wraps around the same Clifford torus $T_{2n-1,1}$ along each end of $M$" is to be understood as follows: there exists $R\in\R$ such that $f(\S^{2n-1}\times(\R\setminus[-R,R]))$ can be written as a multigraph over a fixed Clifford torus $T_{2n-1,1}$ and the supremum of the $C^1$-norms of the graph functions among all the leaves tends to zero as $R\to\infty$. 
\end{remark}

The sign changing property of the mean curvature is not a coincidence, as our second main
result points out. Similarly to the case of translators in the Euclidean space, we are able
to characterize the mean convex complete examples in $\S^3$. 

\begin{mythm}\label{THMB}
Let $f:M^{2}\to\S^{3}$ be a complete Hopf soliton without boundary. If there exists a compact set $K$ such that $H$ does not change sign on any component of $M \backslash K$, then $M^{2}$ is isometric to a covering of a minimal Clifford torus $T_{1,1}$.
\end{mythm}

Theorem \ref{THMB} is a consequence of the following more general statement:

\begin{mythm}\label{THMC}
	Let $f:M^{2}\to\S^{3}$ be a complete Hopf soliton without boundary, and let $K \subset M$ be a compact set. If the mean curvature of a  connected component $E$ of $M \backslash K$ does not change sign and is not identically zero, then $E$ is relatively compact.
\end{mythm}

To prove Theorem \ref{THMC},  we perform a conformal change of the metric of $M^2$ by an appropriate power of the mean curvature, and adapt a beautiful technique
pioneered  by Fischer-Colbrie \cite{colbrie} (see also Schoen \& Yau \cite{schoen_yau_cmp}, Shen \& Ye \cite{shen_ye_min}, Shen \& Zhu \cite{shen_zhu_MathAnn} and Catino, Mastrolia \& Roncoroni \cite{catino}) to show that $\overline{E}$ is compact. Unfortunately, our method
seems not sufficient to treat complete Hopf solitons in dimensions greater than two.
It would be interesting to know if there are complete Hopf solitons on one side of the Clifford torus, or
in a closed half-sphere. The latter would be an analogue of Nadirashvili's construction \cite{nad}.

The Clifford torus can also be characterized among Hopf solitons whose second fundamental form satisfies the bound $|A|^ 2 \le 2n$, very much in the spirit of a classical result by Simons \cite{simons}, Lawson \cite{lawson} and Chern, do Carmo \& Kobayashi \cite{chern}.

\begin{mythm}\label{THMD}
Let $f:M^{2n}\to\S^{2n+1}$ be a complete Hopf soliton without boundary, whose second fundamental form $A$ satisfies $|A|^ 2 \le 2n$. Suppose that the norm of the traceless second fundamental form attains a local maximum. Then $M^{2n}$ is isometric to a covering of a minimal Clifford torus $T_{a,b}$ for some $a,b$ with $a+b = 2n$. 
\end{mythm}

\section{Hopf solitons}

In this section we will derive and summarize
the most relevant equations related to Hopf solitons. Before considering this case, let us observe a general fact:

\begin{proposition}\label{prop_div_xiT}
Let $f : M^n \to N^k$ satisfy the soliton equation
\[
\textbf{\textit{H}} =\xi^{\perp}
\]
for some Killing vector field on $N$. Then,
\[\operatorname{div}(\xi^{\top})=H^2\]
holds on $M^n$. As a consequence, compact boundaryless MCF solitons with respect to Killing vector fields must be minimal.
\end{proposition}

\begin{proof}
	Let $\langle \cdot\, , \cdot \rangle$ be the Riemannian metric on $N^k$ and let $\nabla,\overline{\nabla}$ be, respectively, the Levi-Civita connections of $M^n$ and $N^k$. Fix a local Darboux frame $\{e_i;e_\alpha\}$ along $f$, with $e_i$ tangent to $M^n$ and $e_\alpha$ normal to $M^n$. From the identity 
	\[
	\xi ^\top= \xi  - \xi^\perp = \xi - \textbf{\em H}
	\]
	we get
	\begin{eqnarray*}
		\operatorname{div}(\xi^{\top})  &\!\!\!=\!\!\!&\langle \nabla_{e_i}\xi^{\top},e_i\rangle=
		\langle \overline{\nabla}_{e_i}\xi^{\top},e_i\rangle
		=\langle \overline{\nabla}_{e_i}(\xi- \textbf{\em H}),e_i\rangle\\
		&\!\!\!=\!\!\!&\langle \overline{\nabla}_{e_i}\xi,e_i\rangle +  \langle \textbf{\em H}, \overline{\nabla}_{e_i} e_i\rangle = H^2.
	\end{eqnarray*}
The last assertion follows by using the divergence theorem. 
\end{proof}

Hereafter, we consider 
\[
M^{2n}\rightarrow\S^{2n+1}\subset\R^{2n+2}\equiv\C^{n+1}.
\]
We denote by $D$ the Levi-Civita connection of the Euclidean space $\R^{2n+2}$,
by $\overline{\nabla}$ the Levi-Civita connection of $\mathbb{S}^{2n+1}$ and by $\nabla$ the Levi-Civita connection
of the induced metric on $M^{2n}$. To simplify the notation we often denote the Riemannian
metrics on $\R^{2n+2}$, $\S^{2n+1}$ and $M^{2n}$ by the same symbol
$\gind\equiv\langle\cdot\,,\cdot\rangle$. 
Moreover, we denote by $\nu$ the unit normal of $M^{2n} \rightarrow \mathbb{S}^{2n+1}$ and by
$$A = -\overline{\nabla} \nu$$
its corresponding shape operator. The mean curvature vector is defined by
\[
\textbf{\em H}=({\trace}A)\nu.
\]
\subsection{Hopf vector fields}

The space $\R^{2n+2}$ supports many complex structures, i.e. linear isometries $J:\R^{2n+2}\to\R^{2n+2}$ such that
$J^{2}=-I$. Suppose that $J$ is
a complex structure on $\R^{2n+2}$.
The vector field
\begin{equation}\label{hopfvectorfield}
\xi =-Jp,
\end{equation}
where $p$ is the position vector of $\S^{2n+1}$, is globally defined and tangent to the sphere.
The vector field $\xi$ is called the {\em Hopf vector field}. Denote by $\omega$ the $1$-form associated to the vector field $\xi$, i.e.
\[
\omega(X)=\langle X,\xi\rangle, \quad \text{for every}\,\,\,X\in\mathfrak{X}(\S^{2n+1}).
\]
Then, for any vector field $X$ on $\S^{2n+1}$, the decomposition in tangent and normal
components determines a $(1,1)$-tensor field $\phi$
such that
\begin{equation}\label{phi}
JX=(JX)^{\top\S^{2n+1}}+(JX)^{\perp\S^{2n+1}}=\phi(X)+\omega(X)p.
\end{equation}
One can readily check from
\eqref{phi} that $\phi$ and
$\omega$ satisfy the following properties:
\[
\phi(\xi)=0,\quad \omega\circ\phi=0,\quad \omega(\xi)=1\quad\text{and}\quad
\phi^2=-I+\omega\otimes\xi.
\]
Moreover, $\phi$ is skew-symmetric and it is an isometry on the horizontal bundle
\[\mathcal{H}=\big\{X\in\mathfrak{X}(\S^{2n+1}):\langle X,\xi\rangle=0\big\}.\]
Using the Weingarten formula, we easily obtain that
\begin{equation}\label{hopf1}
\overline{\nabla}_X\xi=D_X\xi+\langle X,\xi\rangle p
=-JD_Xp+\langle X,\xi\rangle p
=-JX+\omega(X)p=-\phi(X),
\end{equation}
for each $X\in\mathfrak{X}(\S^{2n+1})$.
Since $\phi(\xi)=0$ it follows that $\xi$ has totally geodesic integral curves.
Moreover,
\[|\phi|^2=|\overline{\nabla}\xi|^2=2n.\]
From
\eqref{hopf1} we deduce that $\xi$ is a Killing vector field, i.e.
\[
\langle\overline{\nabla}_X\xi,Y\rangle+\langle X,\overline{\nabla}_Y\xi\rangle=0,
\]
for any $X,Y\in\mathfrak{X}(\S^{2n+1})$.
Any unit Killing vector field on $\S^{2n+1}$ can be
written in the form \eqref{hopfvectorfield} for some complex structure of $\C^{n+1}$;
see for example \cite[Proposition 3.6]{gil}.
\subsection{Hopf solitons}
Let $M^{2n}\rightarrow\S^{2n+1}$ be an oriented hypersurface of the sphere.
 Let us denote by $(\cdot)^{\top}$ and $(\cdot)^{\perp}$ the
orthogonal projections on the tangent and the normal bundle of $M^{2n}$, respectively.
Assume that the mean curvature vector $\textbf{\em H}$ of $M^{2n}$ satisfies the elliptic equation
\begin{equation}\label{hopfsoliton}\tag{HS1}
\textbf{\em H}=\xi^{\perp},
\end{equation}
where $\xi$ is the Hopf vector field described in \eqref{hopfvectorfield}. Let
\[H=\langle{\textbf{\em H}},\nu\rangle\]
be the scalar mean curvature of the hypersurface.
Then the equation \eqref{hopfsoliton} can be equivalently written in the form
\begin{equation}\label{hopfsolitonsc}\tag{HS2}
H=\langle\nu,\xi\rangle.
\end{equation}
Denote now by
$\varPhi:\S^{2n+1}\times\S^1\to\S^{2n+1}$ the flow generated by $\xi$. Then,
for any fixed $\theta\in\S^{1}$, the map $\varPhi_\theta:\S^{2n+1}\to\S^{2n+1}$, given by
\[\varPhi_\theta(p)=\varPhi(p,\theta),\quad p\in\S^{2n+1},\]
is an isometry. One can easily verify now that if $M^{2n}$ is an oriented hypersurface satisfying
\eqref{hopfsoliton}, then the family
\[
M^{2n}_\theta=\varPhi_\theta(M^{2n}),
\]
where $\theta\in\S^{1}$,
moves up to tangential diffeomorphisms by the MCF in $\S^{2n+1}$. For this reason, hypersurfaces 
satisfying
\eqref{hopfsoliton} are called {\em Hopf solitons}.

Following standard notation in the field, we call the operator
\[
\Delta_{\xi^{\top}}(\cdot)=\Delta(\cdot)+\langle\xi^{\top},\nabla(\cdot)\rangle.
\]
the {\em drifted Laplacian} on $M^{2n}$. In the next lemma we give some important relations between 
the mean curvature $H$, the Weingarten operator $A$ and the Hopf vector field $\xi$.
\begin{lemma}\label{crlemma}
Let $M^{2n}\rightarrow \mathbb{S}^{2n+1}$ be a Hopf soliton. Then the following formulas hold:
\begin{enumerate}
\item[(a)]
The absolute value of $H$ is at most $1$. Moreover,
\begin{equation}\label{xitop}
|\xi^{\top}|^2=1-H^2 . 
\end{equation}
\item[(b)]
The gradient of the scalar mean curvature is given by
\begin{equation}\label{grad-H}
\nabla H=(J\nu)^{\top}-A\xi^{\top}.
\end{equation}
\item[(c)] The drifted Laplacian of the mean curvature is equal to
\begin{equation}\label{lap-H}
\Delta _{\xi ^\top }H=-H|A|^2.
\end{equation}
\item[(d)] The drifted Laplacian of the shape operator and of its squared norm are
\begin{eqnarray}
\Delta _{\xi^\top }A&\!\!\!=\!\!\!& (2n-|A|^2) A-HI+[A,J],\label{lap-A1}\\
\tfrac{1}{2}\Delta_{\xi^{\top}}|A|^2&\!\!\!=\!\!\!&(2n-|A|^2)|A|^2-H^2+|\nabla A|^2.\label{lap-A2}
\end{eqnarray}
where $I$ is the identity on $TM$.
\item[(e)] The drifted Laplacian of the principal curvatures of $M^ {2n}$ (at points where a local smooth principal frame exists) is
\begin{eqnarray}\label{lap-lambda-i}
\Delta _{\xi ^ \top} \lambda _ i = (2n-|A|^ 2 ) \lambda _i - H. 
\end{eqnarray}
\item[(f)] The drifted Laplacian of the squared norm of the traceless operator $\mathring A= A-(H/2n)I$ is
\begin{eqnarray}\label{lap-tracelessA}
\Delta _{\xi ^ \top}|\mathring A |^ 2=2(2n-|A|^ 2)|\mathring A|^ 2 + 2|\nabla \mathring A | ^ 2.
\end{eqnarray}
\end{enumerate}
\end{lemma}
\begin{proof}
Let $\{e_1, \dots , e_{2n}\}$ be a local orthonormal frame along the hypersurface $M^{2n}$ such that
$$\nabla _{e_j}e_i(x_0)=0,\quad\text{for all}\quad i,j\in\{1,\dots,2n\},$$
at a fixed point $x_0\in M^{2n}$.

(a) We have
$$
0\le |\xi^{\top}|^{2}=|\xi|^2-|\xi^{\perp}|^2=1-H^2.
$$

(b) Differentiating the mean curvature with respect to $e_i$ and estimating at $x_0$, we get
\begin{eqnarray*}
e_i(H)  &\!\!\!=\!\!\!&e_i\langle \xi,\nu\rangle=\langle \overline{\nabla} _{e_i}\xi , \nu \rangle + \langle \xi , \overline{\nabla}_{e_i}\nu \rangle\\
&\!\!\!=\!\!\!& \langle D _{e_i}\xi , \nu \rangle + \langle \xi , \overline{\nabla}_{e_i}\nu\rangle
=  \langle -Je_i, \nu \rangle + \langle \xi ^\top , \overline{\nabla} _{e_i}\nu \rangle\\
 &\!\!\!=\!\!\!& \langle e_i, (J\nu)^{\top} \rangle -\langle \xi ^\top,Ae_i \rangle.
\end{eqnarray*}
Therefore,
$$
\nabla H=(J\nu)^{\top}-A\xi^{\top}.
$$
(c) Differentiating $H$ again, using the Weingarten formulas, and estimating at $x_0$, we obtain
\begin{eqnarray*}\label{hess-H}
e_je_i(H)&\!\!\!=\!\!\!& \langle  \overline{\nabla}_{e_j}e_i , (J\nu)^{\top} \rangle +
\langle e_i, \overline{\nabla} _{e_j} (J\nu) ^{\top \mathbb{S}^{2n+1}}\rangle
-\langle\nabla_{e_j}\xi^{\top},Ae_i\rangle-\langle\xi^{\top},(\nabla_{e_j}A)e_i\rangle\\
&\!\!\!=\!\!\!&
\langle e_i, \overline{\nabla} _{e_j} (J\nu) ^{\top \mathbb{S}^{2n+1}}\rangle
-\langle\nabla_{e_j}\xi^{\top},Ae_i\rangle-\langle\xi^{\top},(\nabla_{e_j}A)e_i\rangle\\
&\!\!\!=\!\!\!&
\langle e_i, D _{e_j} J\nu\rangle
-\langle\nabla_{e_j}\xi^{\top},Ae_i\rangle-\langle\xi^{\top},(\nabla_{e_j}A)e_i\rangle\\
&\!\!\!=\!\!\!&
\langle Je_i, Ae_j\rangle
-\langle\nabla_{e_j}\xi^{\top},Ae_i\rangle-\langle\xi^{\top},(\nabla_{e_j}A)e_i\rangle.
\end{eqnarray*}
From the identity
$$\xi ^\top= \xi  -\langle \xi ,\nu \rangle \nu =\xi - H \nu,$$
the Weingarten formula and \eqref{grad-H}, we deduce that at the point $x_0$ it holds
\begin{eqnarray*}
\nabla_{e_j}\xi^{\top}&\!\!\!=\!\!\!&\overline\nabla _{e_j}\xi ^{\top}-\langle Ae_j,\xi^{\top}\rangle\nu
=\overline\nabla _{e_j}\xi - \overline\nabla _{e_j}(H\nu)-\langle Ae_j,\xi^{\top}\rangle\nu\\
&\!\!\!=\!\!\!&(\overline\nabla_{e_j}\xi)^{\top}+(\overline\nabla_{e_j}\xi)^{\perp}-e_j(H)\nu+HAe_j-\langle Ae_j,\xi^{\top}\rangle\nu\\
&\!\!\!=\!\!\!&(\overline\nabla_{e_j}\xi)^{\top}+(\overline\nabla_{e_j}\xi)^{\perp}-\langle J\nu,e_j\rangle\nu+HAe_j.
\end{eqnarray*}
Observe now that
$$
(\overline\nabla_{e_j}\xi)^{\perp}=\langle\overline\nabla_{e_j}\xi,\nu\rangle\nu=\langle -D_{e_j}Jp,\nu\rangle\nu
=\langle -Je_j,\nu\rangle\nu=\langle e_j,J\nu\rangle\nu.
$$
Hence, in view of the equation \eqref{hopf1}, we get that
$$
\nabla_{e_j}\xi^{\top}=-(J e_j)^{\top}+HAe_j.
$$
Combining the last two formulas, we see that at the point $x_0$ it holds
\begin{equation}\label{hess-H}
e_je_i(H)=\langle Je_i, Ae_j\rangle+\langle Je_j, Ae_i\rangle-\langle\xi^{\top},(\nabla_{e_j}A)e_i\rangle
- H \langle Ae_i, Ae_j\rangle.
\end{equation}
Because $J$ is skew-symmetric and $A$ symmetric, we arrive at
$$\Delta H = -\langle \nabla H, \xi ^\perp \rangle -H|A|^2.$$

(d) The results follow from \eqref{hess-H} and Simons' formula
\[
\Delta A=-|A|^2A + HA^{(2)}+ \nabla ^2H + 2nA -HI,
\]
see for instance \cite[page 240]{sm2}.

(e) Let $\{ e_1, \cdots, e_{2n}\} $ be a local orthonormal frame along $M^ {2n}$ that diagonalizes $A$ and let $\lambda _ i$ be such that
$A(e_i, e_i)=\lambda _ i.$
Evaluating equation \eqref{lap-A1} on $(e_i, e_i)$ one gets:
\begin{eqnarray*}
\Delta _{\xi ^ \top}\lambda _ i & \!\!\!=\!\!\! & (2n-|A|^ 2)\lambda _i -H +2\langle Je_i, Ae_i\rangle \\
& \!\!\!=\!\!\! & (2n-|A|^ 2 ) \lambda _ i -H +2\lambda _ i \langle Je_i, e_i \rangle \\
& \!\!\!=\!\!\! &  (2n-|A|^ 2 ) \lambda _ i -H,
\end{eqnarray*}
where we used that
$$\langle Je_i, e_i \rangle = 0.$$

(f) Note that
$$\mathring A^{(2)} = A^{(2)} -\frac{H}{n}A+\frac{H^ 2}{4n^ 2}I,$$
whence we deduce that
$$ |\mathring A|^ 2 =|A|^ 2 - \frac{H^ 2}{2n}.$$
Now recalling \eqref{lap-H}, \eqref{lap-A2} and \eqref{hess-H} we compute 
\begin{eqnarray*}
\frac{1}{2}\Delta _{\xi ^ \top }|\mathring A|^ 2 &\!\!\!=\!\!\!& \frac{1}{2} \Delta _{\xi ^ \top }|A|^ 2 -\frac{1}{2n}\Delta _{\xi ^ \top }H^ 2\\
& \!\!\!=\!\!\! & (2n-|A|^ 2 ) |A|^ 2 -H^ 2 +|\nabla A|^ 2 +\frac{H^ 2|A|^ 2}{2n}-\frac{|\nabla H|^ 2}{2n}\\
& \!\!\!=\!\!\! & (2n-|A|^ 2 ) \Big(|\mathring A|^ 2+\frac{H^ 2}{2n}\Big)-\frac{H^ 2}{2n}(2n-|A|^ 2 ) +|\nabla A | ^ 2 -\frac{|\nabla H|^ 2}{2n}\\
& \!\!\!=\!\!\! & (2n-|A|^ 2 ) |\mathring A|^ 2+|\nabla \mathring A | ^ 2,
\end{eqnarray*}
where we used that
$$|\nabla A|^ 2 =|\nabla \mathring A|^ 2 + \frac{|\nabla H | ^ 2}{2n}.$$
This completes the proof.
\end{proof}

\section{Rotationally symmetric Hopf solitons}\label{rotationally}
In this section we will investigate the existence of complete rotationally symmetric Hopf solitons in the unit sphere.

\subsection{Derivation of the equation}\label{uyz}
Following \cite{carmo},
let us regard the unit sphere $\S^{2n-1}$ as a subset of $\R^{2n}$ and the unit sphere $\S^{2n+1}$ as a subset
of $\R^{2n}\times \R^2$.
Moreover, on $\R^{2n}\times\R^2$ consider the complex structure $J$ given by
$$
J(x_1,x_{2},\dots,x_{2n-1},x_{2n};x_{2n+1},x_{2n+2})=(-x_2,x_{1},\dots,
-x_{2n},x_{2n-1};-x_{2n+2},x_{2n+1}).
$$
We denote by $T_{2n-1,1}$ the Clifford torus defined, in the above coordinates, by 
\begin{equation}\label{eq_clifford}
T_{2n-1,1} = \mathbb{S}^{2n-1} \left( \sqrt{\frac{2n-1}{2n}}\right) \times \mathbb{S}^1 \left( \sqrt{\frac{1}{2n}}\right) \subset \R^{2n} \times \R^2,
\end{equation}
which is invariant under the action of $J$. Let $\delta>0$ be a positive number, consider an immersed profile curve 
\begin{equation}\label{gencurve}
\gamma=(u,y,z) \ : \ I=(-\delta,\delta)\to\S^2\subset\R^3
\end{equation}
parametrized by arc-length, and define the map
\[
f \ : \ M^{2n} \doteq \S^{2n-1} \times I\to\S^{2n+1}
\]
given by
\begin{equation}\label{def_map_f}
f(p,s)=(u(s)p;y(s),z(s)).
\end{equation}
It is clear that $f$ is invariant under the group of isometries $\mathbb{O}(2n)$,
where the action of $A\in\mathbb{O}(2n)$ on $\R^{2n}\oplus\R^{2}$ is defined by
$$(p;q)\mapsto (Ap;q), \quad \text{for all}\quad (p; q)\in\R^{2n}\times\R^2.$$
Suppose that $\{\alpha_1,\dots,\alpha_{2n-1}\}$ is a local orthonormal frame field
on $\S^{2n-1}$. Then,
$$
df(\partial_s)=(u'p;y',z'), \qquad df(\alpha_i)=(u\alpha_i;0,0), \qquad \forall \, i\in\{1,\dots,2n-1\}.
$$
The metric $g$ induced by $f$ on $M$ satisfies
\begin{equation}\label{metricwarped}
g_{ss}=1,\quad g_{si}=0\quad\text{and}\quad g_{ij}=u^2\delta_{ij},\quad\text{for all}\quad
i,j\in\{1,\dots,2n-1\},
\end{equation}
with respect to the frame field $\{\partial_s,\alpha_1,\dots,\alpha_{2n-1}\}$.
Observe that $f$ is an immersion whenever $u(s)\neq 0$. Since the profile curve $\gamma$ is parametrized by arc-length, we have that
\begin{equation}\label{sphere}
u^2+y^2+z^2=1\quad\text{and}\quad (u')^2+(y')^2+(z')^2=1.
\end{equation}
This means that we can locally represent the components of $\gamma$ in the form
\begin{equation}\label{eq_param_spherical}
u(s)=\cos r(s),\quad y(s)=\sin r(s)\sin\vartheta(s)\quad\text{and}\quad
z(s)=\sin r(s)\cos\vartheta(s),
\end{equation}
where $r\in(-\pi/2,\pi/2)$ and $\vartheta\in(0,2\pi)$.
From \eqref{sphere}, we see that
\begin{equation}\label{eqr}
(r')^2+(\vartheta')^2\sin^2 r=1.
\end{equation}
From \eqref{eqr} it follows that we can represent $\vartheta$ in terms of $r$, whence we deduce that
we can represent $y$ and $z$ in terms of $u$.
More precisely, from \eqref{eqr} we have that, away from points where
$u^2=1,$
 it holds
\begin{equation}\label{eq_rprimo}
1\ge |r'|=\frac{|u'|}{\sqrt{1-u^2}}\quad\text{and}\quad |\vartheta'|=\frac{\sqrt{1-u^2-(u')^2}}{1-u^2}.
\end{equation}
Hence, by continuity we deduce that
$
u^2+(u')^2\le 1.
$
One can easily check that
\begin{equation}\label{normalrot}
\nu=((y'z-yz')p;uz'-u'z,u'y-uy')
\end{equation}
is unit and normal along the hypersurface.
Moreover, using \eqref{eq_param_spherical}, we see that the principal curvatures of $M$ in direction $\nu$ are
\begin{equation}\label{eq_eigenvalues}
\lambda_1=\cdots=\lambda_{2n-1}=-\frac{\sqrt{1-u^2-(u')^2}}{u}\quad\text{and}\quad
\lambda_{2n}=\frac{u''+u}{\sqrt{1-u^2-(u')^2}}.
\end{equation}
Using the identities
$$
u^2+y^2+z^2=1\quad\text{and}\quad uu'+yy'+zz'=0,
$$
we deduce that
\begin{equation}\label{eq_uprime}
\langle\xi,\nu\rangle=-\langle Jf,\nu\rangle=-u'.
\end{equation}
Consequently, $M^{2n}\subset\S^{2n+1}$ represents a Hopf soliton if and only if the function $u$ satisfies the following ODE:
\begin{equation}\label{shopf1}\tag{SHS}
uu''+(2n-1)(u')^2+2nu^2-(2n-1)=-uu'\sqrt{1-u^2-(u')^2};
\end{equation}
compare with \cite[eq. (3.13)]{carmo}.
Note that \eqref{shopf1} can be written, equivalently, in the form
\begin{equation*}
uu''-(2n-1)(1-u^2-(u')^2)+uu'\sqrt{1-u^2-(u')^2}+u^2=0.
\end{equation*}
The constant solution of \eqref{shopf1} corresponds to the Clifford torus $T_{2n-1,1}$ in \eqref{eq_clifford}.
The equation \eqref{shopf1} is equivalent to the system of first order ODEs:
\begin{equation}\label{pb}
\begin{cases}
     u' = P(u,v)=v,\\
	v' =Q(u,v)=(2n-1)u^{-1}(1-u^2-v^2)-v\sqrt{1-u^2-v^2}-u,
\end{cases}
\end{equation}
which we regard as defined in the domain
$$
D=\{(u,v)\in\R^2:u^2+v^2< 1\,\,\,\text{and}\,\,\,u>0\}.
$$
Note that $D$ is the largest domain where the vector field
$$(u,v)\mapsto F_{(u,v)}=(P(u,v),Q(u,v))$$
is $C^1$-smooth. Observe that a solution to \eqref{pb} generates a soliton via the prescription $f$ in \eqref{def_map_f} once the
initial value for $\vartheta$ is prescribed. In what follows we shall always consider $\vartheta(0)=0$ so that
the map $f$ is uniquely determined by $\varrho=(u,v)$. We call such $f$ the {\em soliton associated to $\varrho$}.

\subsection{First properties of the solutions} Observe that
\begin{equation}\label{un}
p_n=(u_n,v_n)=\big(\sqrt{1-1/2n},0\big)
\end{equation}
is the only equilibrium point of \eqref{pb}, and from \eqref{def_map_f} it corresponds to $T_{2n-1,1}$ in \eqref{eq_clifford}. Moreover,
note that
$$
dF_{(u_n,v_n)}=
\begin{pmatrix}
0 & 1\\
-4n & -1/\sqrt{2n}
\end{pmatrix}
$$
with characteristic values
\begin{equation}\label{roots}
\sigma_{1,2}=\alpha\pm i\beta=\frac{-1\pm i\sqrt{32n^2-1}}{2\sqrt{2n}}.
\end{equation}
This means that $p_n=(u_n,v_n)$ is a stable equilibrium point; i.e. 
the integral curves of the vector field $F$ are accumulating to $p_n$ as a
spiral; see \cite[p. 111]{hur}.
Consider the maximal unique solution
$$
\varrho \ : \ I \doteq (s_1,s_2) \to D, \qquad \varrho(s)=(u(s),v(s))
$$
of the system \eqref{pb} issuing from $\varrho(0) \neq p_n$. Define, along $\varrho$, the functions $g,\zeta:I\to D$ given by
$$
g=\sqrt{1-u^2-v^2}\qquad\text{and}\qquad \zeta=u^{2n-1}g.
$$
From \eqref{pb} and direct computation, we see that
\begin{equation}\label{eq_g}
g'+(2n-1)\frac{v}{u}g=v^2,
\end{equation}
whence, by an integration, we infer the following monotonicity formula, which will be crucial for our investigation:
\begin{equation}\label{int1}
\zeta' = (u^{2n-1}g)'=u^{2n-1}v^2\ge 0.
\end{equation}
As a consequence, we easily obtain the behavior of $\varrho$ for positive values of $s$.

{\bf Fact 1:} {\em Any solution $\varrho$ (regardless to the initial condition in $D$) is defined for all positive times and tends to $p_n$ spiraling around it as $s \to \infty$.}

{\em Proof:} A first consequence of \eqref{int1} is that $\varrho$ cannot be periodic, i.e. there are no periodic solutions of the dynamical system \eqref{pb}. Integrating \eqref{int1}, we see that
$$
u^{2n-1}(s)g(s) \ge u(0)^{2n-1}g(0) \doteq c >0, \quad\text{for}\quad s>0.
$$
Hence, $\varrho([0,s_2))$ is contained in the compact set $\{\zeta \ge c\} \subset D$. Hence, $s_2 = \infty$ and by the
classical Poincar\'e-Bendixson Theorem (see, for example, \cite{hur}), $\varrho$ either converges to a cycle or tends to the equilibrium point $p_n$. Since there are no periodic solutions, the existence of a cycle is excluded. The behaviour as a spiral around $p_n$ is guaranteed by the local behaviour around a stable equilibrium point with roots \eqref{roots}. \hfill$\circledast$

\subsection{Construction of solutions issuing from points of $\partial D \cap \{u>0\}$}

In this paragraph we construct solutions of \eqref{pb} issuing from any fixed point $q\in \partial D \cap \{u>0\}$. The key point is the validity of the following property.

{\bf Fact 2: } {\em Let $q \in \partial D \cap \{u>0\}$ and assume that there exist $\eps, \delta > 0$ and solutions $\varrho_j = (u_j,v_j) : [0,\delta] \to D$ to \eqref{pb} satisfying:
\[
\varrho_j(0) \to q \quad \text{and}\quad \varrho_j([0, \delta]) \subset \{u \ge \eps\}. 
\]
Then, there exists $\hat\varrho : [0,\delta] \to \overline{D}$ such that, up to a subsequence, $\varrho_j \to \hat\varrho$ in $C^\alpha([0,\delta])$ for each $\alpha \in (0,1)$. Moreover, $\hat\varrho \in C^1([0,\delta])$
solves \eqref{pb} on $[0,\delta]$ with
$\hat\varrho(0) = q$. Furthermore, $\hat\varrho(s) \in D$ for $s > 0$, and it gives rise to a Hopf soliton. 
}

{\em Proof:} Since the images of $\varrho_j$ are contained in $\{u \ge \eps\}$, then by \eqref{pb} $\varrho_j$ are uniformly bounded in $C^1([0,\delta])$. Thus, by Ascoli-Arzel\`a, they converge in $C^\alpha([0,\delta])$ up to subsequences (not relabelled)  to some $\hat\varrho : [0,\delta] \to \overline{D}$ such that $\hat\varrho(0) = q$. Write $\varrho_j = (u_j,v_j)$ and $\hat\varrho = (u,v)$. As uniform limit of solutions to \eqref{pb}, $\hat\varrho$ solves \eqref{pb}, hence $\varrho \in C^1([0,\delta])$. To prove that $\hat\varrho(s) \in D$ for $s > 0$, and consequently that $\hat\varrho$ generates a Hopf soliton, we proceed as follows: to each $\varrho_j$, we associate functions $y_j$ and $z_j$, as explained in Section \ref{uyz}. Consider now
the immersion $f_j: \mathbb{S}^{2n-1} \times [0, \delta] \to \mathbb{S}^{2n+1}$ given by
\[
f_j(p,s) = (u_j(s)p;y_j(s),z_j(s)). 
\]
By using \eqref{eq_eigenvalues}, \eqref{pb} and \eqref{sphere}, and since $\varrho_j([0, \delta]) \subset \{u \ge \eps\}$, the second fundamental form of $f_j$ is uniformly bounded and so are $u_j',y_j',z_j',u_j''$. It follows that, up to a subsequence, $f_j$ converges in $C^{1,\alpha}_{\rm loc}$ to a limit map $f_{\infty}:\S^{2n-1}\times(0,\delta)\to\S^{2n+1}$ given by
$$
f_\infty(p,s) = (u_\infty(s)p;y_\infty(s),z_\infty(s)).
$$
In particular, $u_\infty = u$ and we suppress the subscript $\infty$ for convenience. 
Since each $f_j$ is a Hopf soliton, by elliptic regularity for solutions to the PDE describing \eqref{hopfsolitonsc}, 
so is $f$. Assume by contradiction the existence of $\bar s > 0$ so that $\hat\varrho(\bar s) \in \partial D$. 
Define
$$\zeta_j = u_j^{2n-1}g_j\quad\text{and} \quad\zeta = u^{2n-1}g.$$
From \eqref{int1} we deduce that each $\zeta_j$ is 
monotone, hence so is $\zeta$. Therefore,
$$0 = \zeta(0) = \zeta(\bar s)$$
imply that $\zeta \equiv 0$ on $[0,\delta]$, 
thus $\hat\varrho([0,\bar s]) \subset \partial D$. By using \eqref{eq_param_spherical}, \eqref{eqr} and \eqref{eq_rprimo}, it turns out that for $f$ it holds $r' \equiv 1$ and $\vartheta' \equiv 0$ on $[0,\bar s]$.
Thus $f:\S^{2n-1}\times(0,\bar{s})\to\S^{2n+1}$ is given by
\[
f(p,s) = \left(p \cos\left( s + s_0\right), \sin\left( s + s_0\right) \sin \vartheta_0, \sin\left( s + s_0\right) \cos \vartheta_0 \right),
\]
for some constants $s_0$ and $\vartheta_0$. Therefore, the map $f$ represents a piece of a totally geodesic sphere. However, by \eqref{eq_uprime} it follows that for all $s \in (0,\bar s)$ it holds
\[
0 = H = -u' = \sin(s+s_0),
\]
contradiction. Having shown that $\hat\varrho(s) \in D$ for $s > 0$, and that the limit map $f$ is a Hopf soliton, the proof is complete. \hfill$\circledast$

{\bf Fact 3: } {\em For any fixed $q\in\partial D\cap\{u>0\}$ there exists a solution $\varrho:(0,\infty)\to D$
to \eqref{pb} that extends $C^1$ at $s=0$ with $\varrho(0)=q$.}

{\em Proof:} Consider a sequence of points $q_j\in D$ with $q_j\to q$, and fix $\eps>0$ so that $q_j\in\{u\ge 2\eps\}$ for each $j$. Let $\varrho_j:[0,\infty)\to D$ solve \eqref{pb} with initial data $q_j=\varrho_j(0)$. Note that $\varrho_j$ is defined on the entire half line by Fact 1.  Moreover, from the
bound on $\varrho_j'$ guaranteed by \eqref{pb} on $\{u\ge\eps\}$ we infer the existence of
$\delta=\delta(\eps)$ such that $\varrho_j([0,\delta])\subset \{u\ge\eps\}.$ Applying Fact 2, we construct a
limit curve $\varrho:(0,\delta]\to D$ with the desired properties. Additionally, this limit curve can be extended to $(0,\infty)$ by Fact 1. \hfill$\circledast$

We point out that the analysis of the solutions of \eqref{pb} that we have performed so far is already enough to conclude the proof of Theorem \ref{THMA}. Nonetheless, we now proceed to a detailed study of some of the features of the generic solution to \eqref{pb}, which we deem of independent interest. The proof of Theorem \ref{THMA} shall be presented at the end of the next subsection.

\subsection{Further properties of solutions to \eqref{pb}}

In this paragraph we shall prove that every solution of \eqref{pb} issues from a point of $\{u^2+v^2=1\}$
at some finite time. Hereafter, let $\varrho:(s_1,\infty)\to D$ be a solution to \eqref{pb} with
initial condition $\varrho(0)=q\in D\setminus\{p_n\}$, defined on a maximal time
interval. Observe that $\varrho$ does not correspond to the Clifford torus.

{\bf Fact 4:} {\em The curve $\varrho$ escapes every compact set of $D$ as $s\to s_1$, and $\lim_{s\to s_1}\zeta=0$. In other
	words, $\varrho$ approaches in Hausdorff distance $\partial D=\{\zeta=0\}$ as $s\to s_1$.}

{\em Proof:} If $s_1>-\infty$ the claim follows by general ODE theory. Assume that $s_1=-\infty$ and to the contrary
that
$$\overline{\varrho((-\infty,0))}\subset D.$$
By the Poincar\'e-Bendixson Theorem, and since there are no limit cycles, it follows
that $\varrho\to p_n$ as $s\to-\infty$. However, in this case
$$
\lim_{s\to-\infty}\zeta =\lim_{s\to\infty}\zeta,
$$
hence $\zeta$ is constant because of its monotonicity. From $0=\zeta'=u^{2n-1}v^2$, we deduce that
$v\equiv 0$. By \eqref{pb}, $\varrho$ coincides with $p_n$ contradicting our assumption. Again by monotonicity $\zeta$ must tend to zero as $s\to s_1$, since otherwise the curve $\varrho$
stays within a compact set, something we already excluded. \hfill$\circledast$

In the sequel we will exclude the possibility that $\varrho$ wraps around the boundary of $D$. To this
end, consider polar coordinates centered at $p_n$, i.e.
$$
u=u_n+r\cos\theta\quad\text{and}\quad v=r\sin\theta.
$$
A direct computation gives
\begin{equation}\label{theta}
\theta'=-1+(2n-1)\frac{g^2}{u}\frac{\cos\theta}{r}-g\sin\theta\cos\theta-u_n\frac{\cos\theta}{r}.
\end{equation}

{\bf Fact 5:} {\em There exists $s_0 > s_1$ such that $\theta' < -1/3$ on $(s_1,s_0)$.}

{\em Proof:} By Fact 4, for fixed $\eps$ there exists $s_\eps$ such that 
\[
\rho((s_1,s_\eps)) \subset \{g < \eps\} \cup \{u< \eps^3\}. 
\]
On $\{g<\eps\}$, we have 
\[
1- \eps^2 < (u_n + r\cos \theta)^2 + r^2 \sin^2 \theta = r^2 + u_n^2 + 2r u_n \cos \theta,
\]
and thus
\begin{equation}\label{eq_good_theta}
	\frac{u_n \cos \theta}{r} > \frac{1-2n\eps^2}{4nr^2} - \frac{1}{2}.
\end{equation}
On the subset $\{g<\eps, u \ge u_n\}$, inequalities 
\[
0 \le \frac{\cos\theta}{r} <  \frac{1}{1-u_n-\eps}, 
\]
hold, and thus from \eqref{theta} we deduce
\[
\theta' < -1 + (2n-1)\frac{\eps^2}{u_n} \frac{1}{1-u_n-\eps} +\eps < - \frac{1}{3},
\]
if $\eps$ is small enough. On the subset $\{g<\eps, u < u_n\}$, we have $\cos\theta < 0$, hence we discard the term with $g^2/u$ and use \eqref{eq_good_theta} to deduce that for  small enough $\eps$, we have
\[
\theta' < -1 + \eps - \frac{1-2n\eps^2}{4nr^2} + \frac{1}{2} < - \frac{1}{3}. 
\]
Eventually, on $\{u<\eps^3, g \ge \eps\}$ and for small enough $\eps$ we have
\[
(2n-1) \frac{g^2}{u} - u_n \ge (2n-1) \frac{\eps^2}{\eps^3} - u_n > 0 \quad\text{and}\quad r\cos\theta < \eps - u_n < 0,
\] 
and thus from
$|g\cos\theta\sin\theta| < 1/2$
we deduce
\[
\theta' < -1 + \frac{1}{2} < -\frac{1}{3}.
\]
This concludes the proof. \hfill$\circledast$

Set
$$
\theta^*=\lim_{s\to s_1}\theta.
$$

{\bf Fact 6:} {\em $\theta^* < \infty$.}

{\em Proof:} Assume to the contrary that the converse is true. Then, from Facts 4 and 5, the curve
$\varrho$ spirals towards $\partial D$ as $s \to s_1$, with limit set the entire $\partial D$. Moreover, observe that for each $c>0$ the derivative $\theta'$ is bounded on sets of the form
$\{g<\eps\} \cap \{u > c\}.$
Therefore, if $\theta^* = \infty$ then necessarily $s_1 = -\infty$. Let $p = (1/\sqrt{2},1/\sqrt{2})$, and pick a sequence $s_j \to -\infty$ such that $\varrho(s_j) \to p$. Fix a small $\eps>0$ and consider the region
$$V = \{g < \eps\} \cap \{|\psi - \pi/4|< \eps\},$$
where here $\psi$ is the angle measured from the origin of $\R^2$. Since $|\theta'|$ is bounded on $\varrho^{-1}(V)$, then there exists $\delta$ such that, for each $j\in\natural{}$, the interval $[s_j,s_j+\delta]$ is mapped by $\varrho$ to $V$. Defining
$\varrho_j : [0, \delta] \to D$ by
$
\varrho_j(s) = \varrho(s-s_j),
$
we can apply Fact 2 to deduce that $\varrho_j \to \hat{\varrho}$ uniformly on $[0,\delta]$, where $\hat{\varrho}(s) \in D$ for $s \neq 0$. The image of $\hat{\varrho}$ would therefore be part of the limit set of $\varrho$, contradicting Fact 4. \hfill$\circledast$

Having proved that $\theta^* < \infty$, by Fact 5 necessarily $s_1 > -\infty$. We continue with the
following:

{\bf Fact 7:} {\em The curve $\varrho$ converges to some point $(u_1,v_1)\in\partial D$ as $s\to s_1$.}

{\em Proof:} The conclusion follows since $\theta$ has a finite limit $\theta^*$ as $s \to s_1$, $\partial D$ 
is star-shaped with respect to $p_n$ and $\varrho$ is approaching $\partial D$ by Fact 4.
\hfill$\circledast$

{\bf Fact 8:} {\em It holds $u_1^2 + v_1^2 = 1$.}

{\em Proof:} Assume, to the contrary, that $u_1^2 + v_1^2  < 1$. From $(u_1,v_1) \in \partial D$ we deduce $u_1=0$, $|v_1|<1$. By Fact 5, we have that $\theta' < 0$ on $(s_1,s_0)$. Thus $v > 0$ in a neighbourhood $(s_1,s_*) \subset (s_1,0)$, where we choose $s_*$ to be the first time after $s_1$ where $v=0$. Hence, $v_1 \ge 0$. We differentiate $\log u$ twice to get
\[
(\log u)'' + g (\log u)' = \frac{2n g^2 -1}{u^2}.
\]
Integrating on $(s,s_*)$ and using $v(s_*)=0$, we get that for $s \in (s_1,s_*)$ it holds
\begin{equation}\label{eq_logu}
0 \le (\log u)' = \exp\left(\int^{s_*}_s g(t)\,dt \right) \int_s^{s_*} \frac{1-2n g^2(t)}{u^2(t)} \exp\left( -\int_t^{s_*} g(\tau)\,d\tau \right)\, d t.
\end{equation}
Since $\varrho$ has a limit, it follows that
$$g(s) \to g_1 = \sqrt{1-v_1^2},\quad \text{as}\,\, s \to s_1.$$
Assume that  
$
1-2ng_1^2 < 0.
$
Then, because $|u'(s)| \le 1$, the right-hand side of the equation \eqref{eq_logu} diverges to $-\infty$ as $s \to s_1$, a contradiction. Hence,
$$
1- 2ng_1^2 \ge 0
$$
and, in particular $v_1 > 0$. By L'H\^opital's Rule and since
$u(s) \sim v_1(s-s_1),$ as $s \to s_1,$
\[
\begin{array}{ll} 
\disp \int_s^{s_*} \frac{1-2n g(t)^2}{u(t)^2} \exp\left( -\int_t^{s_*} \!\!\!g(\tau)d\tau \right) d t = o \left( \frac{1}{s-s_1} \right) & \!\!\!\!\text{if }  1=2ng_1^2; \\[0.5cm]
\disp \int_s^{s_*} \frac{1-2n g(t)^2}{u(t)^2} \exp\left( -\int_t^{s_*} \!\!\!g(\tau)d\tau \right) d t
\sim \frac{1-2n g_1^2}{v_1^2 (s-s_1)} \exp\left( -\int_{s_1}^{s_*} \!\!\!g(\tau)d\tau \right) & \!\!\!\!  \text{if }
1>2ng_1^2.
\end{array}
\]
On the other hand,
$$(\log u)' = v/u \sim 1/(s-s_1),\quad \text{as}\,\,\, s \to s_1.$$
Plugging into \eqref{eq_logu}, we arrive at the conclusion that
\[
1-2n g_1^2 \ge 0 \quad\text{and}\quad 1 = \frac{1-2n g_1^2}{v_1^2}.
\]
Therefore $v_1 = 1$, which leads to a contradiction. \hfill$\circledast$

{\bf Proof of Theorem \ref{THMA}.}
Consider the solution $\varrho$ in Fact 3 issuing from the point $(1,0)$ and let $f:\S^{2n-1}\times[0,\infty)\to\S^{2n+1}$ be the corresponding unique rotationally symmetric Hopf soliton given as in Section \ref{uyz}.
From Fact 1, the map $f$ tends to the Clifford torus $T_{2n-1,1}$ as $s \to \infty$, and since
$$H = -u'=-v$$
the behaviour of $\varrho$ as a spiral guarantees that $H$ changes sign in each end. Also, $f$ satisfies
$$
f(p,0) = (p;0,0),
$$
so the boundary of $M$ is the round, totally geodesic $\mathbb{S}^{2n-1}$ in $\S^{2n+1}$ which is focal to $T_{2n-1,1}$.
From Fact 3, the function $u$ extends $C^2$ up to $s=0$, and from \eqref{eq_eigenvalues}, \eqref{pb} the principal curvatures of $f$ are zero along $\partial M$. Let us locally represent $M$ in a region around $\partial M$ as a graph over an equator of $\S^{2n+1}$. Then by direct computations we see that the gradient and the Hessian of the height function are controlled by the metric, the unit normal and the second fundamental form of $f$.
Therefore, the map $f$ is $C^2$-smooth up to the boundary and $\partial M$ is totally geodesic in $M$, too. 
Up to a rotation $\tau$ in the plane $(y,z)$, which commutes with $J$ and thus preserves Hopf solitons, we can assume that  $\gamma$ in \eqref{gencurve} satisfies
	\begin{equation}\label{gammas1}
		\begin{array}{lll}
			\gamma(0)&\!\!\!=\!\!\!& \disp (u(0),y(0),z(0))\,\,\,\,=\,\,(1,0,0),
			\\[0.2cm]
			\gamma'(0)&\!\!\!=\!\!\!& \disp (u'(0),y'(0),z'(0))=\,\,(0,1,0),
		\end{array}
	\end{equation}
	so that
	$\nu(p,0) = (0;0,-1).$
	Consider the reflection $R$ of $\S^{2n+1}$ given by 
	$$
	R(x_1,\dots,x_{2n};x_{2n+1},x_{2n+2})=(x_1,\dots,x_{2n};-x_{2n+1},x_{2n+2}),
	$$
	which fixed $\partial M$ and $\nu(p,0)$. By a direct  computation we see that $R\circ f$ is also a rotationally symmetric solution to \eqref{eq_uprime}. Observe that the rotationally symmetric hypersurfaces given by $f$ and $R\circ f$ meet along the unit sphere
	$$
	\S^{2n+1}\cap\{x_{n+1}=0=x_{n+2}\}.
	$$
	From \eqref{gammas1}, \eqref{normalrot}, \eqref{pb}, and \eqref{eq_eigenvalues}, we see that the
	map
	$\hat f: \mathbb{S}^{2n-1} \times\R \to \mathbb{S}^{2n+1},$
	given by
	\[
	\widehat f(p,s) = \left\{ \begin{array}{ll}
	f(p,s), & \text{if } \, s \ge0, \\[0.2cm]
	R \circ f(p,-s), & \text{if } \, s \le0,
	\end{array}
	\right.
	\]
	is a $C^2$-smooth hypersurface which solves \eqref{eq_uprime}.
	Elliptic regularity implies that $\hat f$ is even $C^{\infty}$-smooth. Fact 1 and \eqref{pb} guarantee that $u(s) \to u_n$ and $u'(s) \to 0$ as $s \to \infty$. Since $T_{2n-1,1}$ is invariant by $R$ and by the rotation $\tau$,
	the map $\hat{f}$ therefore satisfies the properties in Remark \ref{rem_wraps}, i.e. it wraps around $T_{2n-1,1}$ both as $s \to \infty$ and as $s \to -\infty$. Concluding, by \eqref{metricwarped}
	the metric induced by $\hat f$ is 
	$$
	g = ds^2 + \hat{u}(s)^2g_{_\mathbb{S}},
	$$
	with $g_{_\mathbb{S}}$ the round metric on $\mathbb{S}^{2n-1}$ and $\hat{u}$
	a positive smooth function globally defined on $\R$. Therefore, by
	\cite[Lemma 40, p. 209]{neill}, the metric $g$ is complete.

\section{Rigidity theorems}
In this section we give the proofs of our main rigidity results.

\subsection{Proof of Theorem \ref{THMC}} Without loss of generality we may assume that $H\ge 0$ on $E$. By \eqref{lap-H} and the strong maximum principle, either $H \equiv 0$ on $E$ or $H>0$ everywhere. The first case does not occur by assumption. Up to removing a compact neighbourhood of $\partial E$ with smooth boundary, we can therefore assume that $\partial E$ is smooth and that $H>0$ on $\overline{E}$.

Recall that the Gauss curvature $K$ of $M^2$ is given
by the formula
$$
K=1-\frac{|A|^2}{2}+\frac{H^2}{2}.
$$
The idea of the proof is inspired by \cite{colbrie}. Consider the metric
$$\overline{\gind}=H^{2\beta}\gind,$$
where $\beta > 0$ is a constant to be determined. We shall prove that $\overline{E}$ is compact. Assume the converse. Then, by a direct adaptation of \cite[Lemma 2.2]{mmrs} to manifolds with a compact boundary (simply replace the chosen origin $o$ there by $\partial E$) we can construct a ``shortest ray'' $\gamma : [0,T) \to \overline{E}$ with $\gamma(0) \in \partial E$ and the following properties:
\begin{enumerate}
	\item[$(i)$] $\gamma$ is a divergent curve parametrized by $\bar g$-unit speed, and $\gamma((0,T)) \subset E$;
	\smallskip
	\item[$(ii)$] $\gamma$ is $\overline{\gind}$-minimizing between any pair of its points;
	\smallskip
	\item[$(iii)$] $(M, \overline{\gind})$ is complete if and only if $T = \infty$.
\end{enumerate}

We will reach a contradiction by showing that, for some $\beta$, $(\overline{E}, \overline{\gind})$ is complete and its sectional curvature is bounded below by a positive constant $c^2$. Indeed, by Bonnet-Myers' argument, in this case the diameter of  $(\overline{E},\overline{\gind})$ does not exceed $\pi/c$, hence $\overline{E}$ should be compact against our assumption.

Keeping in mind \eqref{lap-H}, we see that
the Gauss curvature of the $\overline{\gind}$-metric is given by
\begin{eqnarray*}
\overline{K}&\!\!\!=\!\!\!&H^{-2\beta}\big(K-\Delta\log H^{\beta}\big)\\
&\!\!\!=\!\!\!&H^{-2\beta}\left\{1+\Big(\beta-\frac{1}{2}\Big)|A|^2+\frac{H^2}{2}+
\beta\frac{\langle\xi^{\top},\nabla H\rangle}{H}+\beta\frac{|\nabla H|^2}{H^2}\right\}.
\end{eqnarray*}
Using formula \eqref{xitop} and Young's inequality, we get
\begin{eqnarray*}
\beta\frac{\langle\xi^{\top},\nabla H\rangle}{H}&\!\!\!\ge\!\!\!&
 -\frac{\beta}{2}\frac{2|\xi^{\top}|\,|\nabla H|}{H}
 \ge-\frac{\beta(1-H^2)}{2\delta}-\frac{\beta\delta}{2}\frac{|\nabla H|^2}{H^2},
\end{eqnarray*}
where $\delta$ is a positive constant  such that
\begin{equation}\label{choisebd}
{1}/{4}\le{\beta}/{2}<\delta\le 2.
\end{equation}
Furthermore, note that
$${|H|^{-2\beta}}\ge 1\quad\text{and}\quad |A|^2\ge {H^2}/{2}.$$
From the above estimates and the choice of the constants in \eqref{choisebd}
we deduce that
\begin{eqnarray}\label{gaussnew}
\overline{K}&\!\!\!\ge\!\!\!& \frac{1}{H^{2\beta}}\left\{1-\frac{\beta}{2\delta}+
\Big(\frac{\beta}{2}+\frac{1}{4}+\frac{\beta}{2\delta}\Big)H^2
+\beta\Big(1-\frac{\delta}{2}\Big)\frac{|\nabla H|^2}{H^2}\right\}\\
&\!\!\!\ge\!\!\!&1-\frac{\beta}{2\delta}+
\Big(\frac{\beta}{2}+\frac{1}{4}+\frac{\beta}{2\delta}\Big)H^2
+\beta\Big(1-\frac{\delta}{2}\Big)\frac{|\nabla H|^2}{H^2}.\nonumber\\
&\!\!\!\ge\!\!\!&1-\frac{\beta}{2\delta}\nonumber\\
&\!\!\!>\!\!\!&0.\nonumber
\end{eqnarray}
Denote by $s$ and $\overline{s}$, respectively, the $\gind$- and $\overline{\gind} $-arclengths of $\gamma$. Since $\gamma$ is divergent and $(\overline{E},\gind)$ is complete, in $\gind$-arclength the curve $\gamma$ is parametrized for $s \in [0,\infty)$. By property $(iii)$, to check that $(\overline{E},\overline{\gind})$ is complete it is enough to show that
$$
T = \int_0^{\infty}H^{\beta}\, d s=\infty.
$$
Since $\gamma$ is a minimizing $\overline{\gind}$-geodesic, the second variation gives that 
\begin{equation}\label{need1}
\int_{0}^{T}\big\{(\varphi_{\overline{s}})^2-\overline{K}\,\varphi^2\big\}
\,d\overline{s}\ge 0,\quad
\text{for every}\,\,\,\,\varphi\in C^{\infty}_0\big([0,T)\big),
\end{equation}
where $C^{\infty}_0\big([0,T)\big) = \{ \varphi \in C^\infty_c([0,T)) : \varphi(0) = 0\}$. Observe that
$$
\partial_{\overline{s}}=H^{-\beta}\partial_s, \qquad d\overline{s}=H^{\beta}ds, \qquad \varphi_{\overline{s}}
=H^{-\beta}\varphi_s.
$$
Setting $\varphi'=d\varphi/ds$,
the inequality \eqref{need1} becomes
\begin{equation}\label{need2}
\int_{0}^{\infty}\big\{(\varphi')^2H^{-\beta}-\overline{K}H^{\beta}\varphi^2\big\}\,ds\ge 0,\quad
\text{for every}\,\,\,\,\varphi\in C^{\infty}_0\big([0,\infty)\big).
\end{equation}
Set now $\varphi=H^{\beta}\psi$, where $\psi$ is a test-function to be determined later.
Then
$$
\varphi'=\beta H^{\beta-1}H'\psi+H^{\beta}\psi'.
$$
From \eqref{need2} and \eqref{gaussnew}, we obtain that
\begin{eqnarray}\label{need3}
&&\int_0^{\infty}\big\{H^\beta(\psi')^2+\beta^2H^{\beta-2}(H')^2\psi^2+
2\beta H^{\beta-1}H'\psi\psi'\big\}\,ds\nonumber\\
&&\quad\,\,\,\,\,\,\ge\int_0^{\infty}H^{\beta}\psi^2\left\{1-\frac{\beta}{2\delta}+
\Big(\frac{\beta}{2}+\frac{1}{4}+\frac{\beta}{2\delta}\Big)H^2
+\beta\Big(1-\frac{\delta}{2}\Big)\frac{(H')^2}{H^2}\right\}\,ds.
\end{eqnarray}
Choosing
\begin{equation}\label{choisebd1}
\delta=2-2\beta,\quad \delta>2/5\quad\text{and}\quad \beta<4/5,
\end{equation}
all the conditions given in \eqref{choisebd} are satisfied. Moreover,
the inequality \eqref{need3} becomes
\begin{equation}\label{need5}
\int_0^{\infty}\big\{H^{\beta}(\psi')^2+2\beta H^{\beta-1}H'\psi\psi'\big\}\,ds\ge
\frac{4-5\beta}{2\delta}\int_0^{\infty}H^{\beta}\psi^2\,ds.
\end{equation}
Fix a constant $r>1$ and choose now a test-function $\psi$ of the form
$$\psi(s)=s\varrho(s),\quad s\in[0,\infty),$$
where $\varrho\in C^{\infty}_c([0,\infty))$ is a decreasing positive function which satisfies
$$
\left\{ \begin{array}{clcl}
\varrho(s)&\!\!\!=\!\!\!&1& \mbox{for $s\in[0,r]$},\\
\vspace{2pt}
\varrho(s)&\!\!\!=\!\!\!&0 & \mbox{for $s\in[2r,\infty)$},\\
\vspace{2pt}
|\varrho'(s)|&\!\!\!\le \!\!\!&c/r& \mbox{for $s\in[r,2r]$},\\
\vspace{2pt}
|\varrho''(s)|&\!\!\!\le \!\!\!&c/r^{2}& \mbox{for $s\in[r,2r]$}.
\end{array}
\right.
$$
for some constant $c>0$ independent of $r$. Then
$$\psi'=\varrho+s\varrho'\quad\text{and}\quad\psi''=2\varrho'+s\varrho''.$$
From equation \eqref{need5} and integration by parts, we get
\begin{eqnarray*}
\frac{4-5\beta}{2\delta}\int_{1}^{r}H^{\beta}\,ds
&\!\!\!\le\!\!\!&\frac{4-5\beta}{2\delta}\int_{1}^{r}H^{\beta}s^2\varrho^2\,ds\\
&\!\!\!\le\!\!\!&\frac{4-5\beta}{2\delta}\int_{0}^{2r}H^{\beta}s^2\varrho^2\,ds
=
\frac{4-5\beta}{2\delta}\int_{0}^{2r}H^{\beta}\psi^2\,ds\\
&\!\!\!\le\!\!\!&
\int_{0}^{2r}\big\{H^{\beta}(\psi')^2+(H^{\beta})'(\psi^2)'\big\}\,ds
=\int_{0}^{2r}H^{\beta}\big\{(\psi')^2-(\psi^2)''\big\}\,ds\\
&\!\!\!=\!\!\!&\int_{0}^{2r}H^{\beta}\big\{-2\psi\psi''-(\psi')^2\big\}\,ds
\le-\int_{0}^{2r}2H^{\beta}\psi\psi''\,ds\\
&\!\!\!\le\!\!\!&\int_r^{2r}H^{\beta}|2s\varrho\varrho'+2s^2\varrho\varrho''|\,ds\\
&\!\!\!\le\!\!\!& 4c\int_r^{2r}H^{\beta}\,ds.
\end{eqnarray*}
Consequently,
$$
\frac{4-5\beta}{2\delta}\int_{1}^{r}H^{\beta}\psi^2ds\le
 4c\int_r^{2r}H^{\beta}ds
 \le
 4c\int_r^{\infty}H^{\beta}ds.
$$
If the right-hand side is finite, then by letting $r \to \infty$ we reach a contradiction. Therefore, 
$$
\int_1^{\infty}H^{\beta}ds=\infty,
$$
hence $\overline{\gind}$ is complete, and this completes the proof.

\subsection{Proof of Theorem \ref{THMB}}

We begin with the following:

\begin{proposition}\label{prop_clifford}
	Assume that $f : M^2 \to \mathbb{S}^3$ is a complete Hopf soliton without boundary. If there exists an open subset $U \subset M$ so that $H$ does not change sign and vanishes somewhere on $U$, then $f$ is a covering of a Clifford torus $T_{1,1}$.	
\end{proposition}

\begin{proof}
	By \eqref{lap-H} and Hopf's strong maximum principle, we get that $H\equiv 0$ on $U$, hence on the entire $M$ by unique continuation. This implies that $\xi$ is everywhere tangent to $M^2$
	and, in fact, it gives rise to a
	Killing vector field on $M^2$. According to \cite[Theorem 8.2.2]{petersen} the Ricci curvature of $M^2$ in the direction of $\xi $ is given by 
	$$
	{\rm Ric}(\xi, \xi)=  |\nabla\xi|^2-\frac{1}{2}\Delta |\xi|^2.
	$$
	This implies that the Gauss curvature $K$ of the surface satisfies
	\begin{equation}\label{Gauss1}
		K=|\nabla\xi|^2\ge 0.
	\end{equation}
	Consequently, the manifold $M^2$ is parabolic, see \cite{huber}.
	On the other hand, from the Gauss equation we have that
	\begin{equation}\label{Gauss2}
		2K=2-|A|^ 2.
	\end{equation}
	Combining \eqref{Gauss1} and \eqref{Gauss2}, we conclude $M$ is a complete minimal
	parabolic surface with
	$$|A|^ 2\le 2.$$
	In this case Simons' equation becomes
	$$
	\Delta|A|^ 2= 2(2-|A|^ 2 ) |A|^ 2 + 2|\nabla A | ^ 2\ge 0.$$
	From parabolicity we conclude that either
	$|A|^ 2 \equiv 0$ or $|A|^2\equiv 2$.
	If $|A|\equiv 0$, then $f(M^2)$ is a great sphere of $\S^3$. However this possibility cannot occur, because $\xi$ cannot be everywhere
	tangent to $\mathbb{S}^2$. If $|A|^2\equiv 2$, Lawson and Chern-Do Carmo-Kobayashi's rigidity theorem \cite[Theorem 1]{lawson}, \cite{chern} imply that $f(M^2)$ is a
	Clifford torus $T_{1,1}$. As $f$ is a local isometry and $M$ is complete, by Ambrose's theorem $f$ is onto and a Riemannian covering.
\end{proof}

We now conclude the proof of Theorem \ref{THMB}. Let $K$ be the compact set in the statement. If $M$ itself is compact, by Proposition \ref{prop_div_xiT} it follows that $M$ is minimal. Applying Proposition \ref{prop_clifford} we get that $M$ covers $T_{1,1}$. If $M$ is not compact, let $E$ be an end with respect to $K$, that is, a connected component of $M \backslash K$ with noncompact closure. In our assumption, $H$ does not change sign on $E$. By Theorem \ref{THMC}, $H$ should vanish identically on $E$ and we can apply again Proposition \ref{prop_clifford} to conclude that $M$ covers $T_{1,1}$.

\subsection{Proof of Theorem \ref{THMD}}

Suppose that there exists a point $x_0\in M$ such that $|\mathring A|^ 2 (x_0)$ is a maximum for $|\mathring A|^ 2$. According to \eqref{lap-tracelessA} and our hypothesis, we get that 
\begin{equation*}
\lap _{\xi ^ \top }|\mathring A|^ 2 =2(2n-|A|^ 2 )|\mathring A|^ 2 +2|\nabla \mathring A|^ 2 \ge 0,
\end{equation*}
which implies, using the maximum principle, that $|\mathring A|^2 $ must be constant. Back to the above equation we get that the left hand side should be zero, which implies that either $|\mathring A|^ 2 \equiv 0 $, or $|A|^ 2 = 2n $ and $|\nabla \mathring A|^ 2 \equiv 0$. The first case corresponds to a totally umbilic sphere, which by Proposition \ref{prop_div_xiT} must be totally geodesic. However, as $\xi$ should be tangent to $M$, totally geodesic spheres of even dimension cannot be Hopf solitons. It remains to consider the second case. 
From equation \eqref{lap-A1} we get that 
\[
0\equiv \lap _{\xi ^ \top }| A|^ 2=-H^ 2 + |\nabla A | ^ 2 = -H^ 2 + |\nabla \mathring A|^ 2 + \frac{|\nabla H |^ 2}{2n}=-H^ 2 + \frac{|\nabla H |^ 2}{2n},
\]
where we used that
$$|\nabla A|^ 2 = |\nabla \mathring A|^ 2 +\frac{|\nabla H |^ 2}{2n}.$$
The above equation implies that on $M$ it holds
$$H^ 2 \equiv  \frac{|\nabla H |^ 2}{2n}.$$
Using equation \eqref{lap-H} we compute $$\lap _{\xi ^ \top }H^ 2 = -2H^ 2 |A|^ 2 + 2 |\nabla H |^ 2$$ and replacing
$H^ 2 \equiv  |\nabla H |^ 2/(2n)$ and $|A|^ 2 \equiv 2n,$
 above we obtain that
 $ \lap _{\xi ^ \top }H^ 2 \equiv 0,$
 therefore $H^ 2 $ is a harmonic function on $M$. 
Since in this case $|A|^ 2 \equiv 2n$ and
$$|\mathring A|^ 2 =|A|^ 2 -\frac{H^ 2}{2n}$$
attains a maximum, it means that $H^ 2 $ attains a minimum, and being $H^2$ a harmonic function, it means that $H^ 2$ must be constant, and so is $H$. Again from \eqref{lap-H} we have that
$$\lap _{\xi ^ \top }H =-H|A|^ 2,$$
which implies that $H\equiv 0$. Since $|A|^ 2\equiv 2n$, the rigidity results in Lawson \cite{lawson} and Chern, do Carmo and Kobayashi \cite{chern} assure that locally $M$ 
coincides with a Clifford torus $T_{a,b}$ for some $a,b$ with $a+b = 2n$. Since $T_{a,b}$ is the zero set of a suitable real analytic function on
$F: \mathbb{S}^ {2n+1}\to \mathbb{R}$ (see for instance \cite[Example 3, p.194]{nomizu}), and a minimal immersion in the sphere is real analytic as well, the restriction of $F$ 
to $M$ must vanish identically. Hence, the image of $M$ is contained in a single $T_{a,b}$. As $f : M^ {2n} \to T_{a,b}$ is a local isometry and $M$ is complete, Ambrose's Theorem guarantees that $f$ is onto and a Riemannian covering, which proves our claim.

\end{document}